\documentclass[10pt,reqno]{article}

\usepackage{amsmath, amsthm, amssymb}
\usepackage{fullpage}
\usepackage{hyperref}

\usepackage{setspace}
\theoremstyle{plain}
\newtheorem{theorem}{Theorem}
\newtheorem{lemma}[theorem]{Lemma}

\theoremstyle{definition}

\numberwithin{equation}{section}

\newcommand{\keywords}[1]{\textbf{Keywords:} #1}

\newcommand{\eps}{\epsilon}
\newcommand{\uni}[1]{\mathrm{Uniform}[#1]}
\newcommand{\M}{\mathcal{M}}
\newcommand{\R}{\mathbb{R}}
\newcommand{\Z}{\mathbb{Z}}
\newcommand{\Uni}{\mathrm{Uniform}[0,1]}
\newcommand{\Int}{\mathbf{Int}^n}

\newcommand{\pr}[1]{\mathbb{P}\left [ #1 \right]}
\newcommand{\E}[1]{\mathbb{E}\left [ #1 \right ]}
\newcommand{\ind}[1]{\mathbf{1}_{\{#1\}}}

\title{Brownian motion as limit of the interchange process - a direct proof}

\author{Mustazee Rahman \and B\'alint Vir\'ag}

\date{}

\begin{document}

\maketitle

\begin{abstract}
\noindent We prove that the random empirical measure of appropriately rescaled particle trajectories
of the interchange process on path graphs converges weakly to the deterministic measure
of stationary Brownian motion on the unit interval. This is a law of large numbers type result for
particle trajectories of the interchange process. 

After the completion of this manuscript we learned
about a result of Durrett and Neuhauser that implies this result. \footnote{Although their proof contains an error which is not straightforward to fix (see Section \ref{BM}),
the rest of the argument is similar to ours.
For this reason, we will not publish this manuscript in a refereed journal.}
\end{abstract}

\indent \keywords{Interchange process, stationary Brownian motion, hydrodynamic limit}

\section{Introduction} \label{sec:intro}
Let $P_n$ denote the finite path of length $n$, which is the graph on the vertex set $\{1, \ldots, n\}$
with edges between every adjacent integers $i$ and $i+1$. Insert self-loops at the end points
$1$ and $n$ so that every vertex has degree 2. The interchange process on $P_n$ is defined as follows.
Initially, each vertex $i$ in $P_n$ has a particle on it with label $i$. The particles move at random
in continuous time; each of the $n+1$ edges has an independent Poisson process of rate $1/2$
that indicates the times the particles move. An edge \emph{fires} at time $t$ if there is a point on
its Poisson process at time $t$. Whenever an edge fires the two particles along its endpoints are swapped.
We stipulate that the swap occurs instantaneously at the time the edge fires, so if edge $\{i, i+i\}$ fires
at time $t$ then the particle at position $i$ at time $t$ is the particle that was at position $i+1$ just prior to time $t$.
We may assume that all the times along all the different Poisson processes are distinct so that no
two edges fire at the same time. Let $\Int_t(i)$ be the position of particle $i$ at time $t$ on $P_n$.
Then $i \to \Int_t(i)$ is a permutation of $\{1, \ldots, n\}$ and the interchange process is the
permutation valued process $\Int = (\Int_t; t \geq 0)$.

The convention that particles jump instantaneously as edges fire implies that the trajectory
of every particle is a c\`{a}dl\`{a}g path (right continuous with left limits). Having the
edges fire at rate $1/2$ implies that each particle moves according to a rate $1$ simple
random walk on $P_n$. Let
\begin{equation} \label{eqn:trajectory} T^n_i(t) = \frac{\Int_{n^2t}(i)}{n} \quad \text{for}\;\; t \geq 0 \end{equation}
be the rescaled trajectory of particle $i$. Consider the random empirical measure
\begin{equation} \label{eqn:nu}\nu^n = \frac{1}{n} \sum_i \delta_{T^n_i}.\end{equation}

The measure $\nu^n$ is a random Borel probability measure on $D(\R_{+})$,
the space of c\'{a}dl\'{a}g paths from $\R_{+} = [0,\infty)$ into $[0,1]$ with the Skorokhod $J_1$-topology.
Any weak limit of $\nu^n$ is a priori a random measure on $D(\R_{+})$.
We prove that $\nu^n$ converges weakly to a deterministic measure: the law of stationary
Brownian motion on $[0,1]$. This is the law of standard Brownian motion started from a uniform
random point in $[0,1]$ and reflected off of the lines $y = 0$ and $y=1$.

Let us remark that any weak limit point of $\nu^n$ is supported on c\'{a}dl\'{a}g process
$X = (X(t); t \geq 0)$ with the property that for every $t$ the law of $X(t) \sim \Uni$. This follows
from observing that for every $t$ the law of $(1/n) \sum_i \delta_{T^n_i(t)}$ is uniform
on the set $\{i/n; 1 \leq i \leq n\}$ because $\Int_{n^2t}$ is a permutation. Such processes,
called \emph{permuton processes}, arise within the limit theory of permutation-valued
processes \cite{RVV}.

\begin{theorem} \label{thm:interchangeprocess}
Consider the interchange process $\Int$ on a path of length $n$ and let $\nu^n$
from (\ref{eqn:nu}) be the empirical measure of the particle trajectories.
Then $\nu^n$ converges weakly -- as a random probability measure on $D(\R_{+})$ --
to the deterministic measure concentrated on the law of stationary Brownian motion on $[0,1]$.
\end{theorem}

\subsection {Background and Motivation}\label{BM}
Theorem \ref{thm:interchangeprocess} is a law of large numbers phenomenon
for particle trajectories of the interchange process. Several results of this nature
exist in the literature. Durrett and Neuhauser \cite{DN} consider the interchange process
dynamics with finitely many coloured particles. Kipnis and Varadhan \cite{KV}, and
later Rezakhanlou \cite{Reza}, prove such results for the symmetric simple exclusion process
in the context of the hydrodynamic limit.

The main result of Durrett and Neuhauser actually implies Theorem \ref{thm:interchangeprocess}
but unfortunately we were not aware of their work until after the completion of this manuscript.
We should mention that there is an erroneous step in their proof. It is claimed that a simple
random walk on $\Z$ that runs at rate $\eps^{-2}$ has $\eps t^{1/2}$ returns to the origin within time $t$
in expectation (see \cite[equation (2.1)]{DN}). However, the correct order is $\eps^{-1} \, t^{1/2}$.
Our proof essentially presents the
correct argument. Our motivation for this work was to have a direct reference to Theorem
\ref{thm:interchangeprocess} which we use in recent work on permutation limits \cite{RVV}
and is also used by Kotowski and Vir\'{a}g \cite{KV16} to prove a large deviation result for
the interchange process with very asymmetric rates. We will not publish this manuscript
to a refereed journal but will use it as a reference for the aforementioned works.

The main ingredient in the proof of Theorem \ref{thm:interchangeprocess}
is showing that particle trajectories of the interchange process become asymptotically
independent (see Lemma \ref{lem:concentration}). For this we use coupling techniques.
In Section \ref{sec:hydrodynamic} we explain how Theorem \ref{thm:interchangeprocess}
also provides the hydrodynamic limit of the symmetric simple exclusion process.

\section{Tightness and Concentration Criterion} \label{sec:ccc}

The following lemma provides a criterion for $\nu^n$ to be tight.
Let $\M(K)$ denote the space of Borel probability measures on a metric space $K$ in the weak topology.
For a function $f : [0,T] \to [0,1]$ its c\'{a}dl\'{a}g modulus of continuity is
$$m_{f}(\delta) = \inf_{\Pi} \,\max_i \,\sup_{s,t \,\in [t_{i-1},t_i)} |f(t)-f(s)|,$$
where the infimum is over all finite partitions $\Pi = \{0 = t_0 < t_1 < \cdots < t_k = T\}$
satisfying $\min_i |t_i - t_{i-1}| \geq \delta$. A function $f$ is c\'{a}dl\'{a}g if and only if
$m_f(\delta) \to 0$ as $\delta \to 0$.

\begin{lemma} \label{lem:compactcriterion}
Let $f^n_1, \ldots, f^n_n$ be a collection of random paths in $D(\R_{+})$.
Consider the empirical measure of these paths: $\mu^n = \frac{1}{n} \sum_{i=1}^n \delta_{f^n_i}$.
Let $m^n_i(T,\delta)$ be the c\'{a}dl\'{a}g modulus of continuity of $f^n_i$ restricted to $[0,T]$. Suppose
for every $T < \infty$
\begin{equation} \label{eqn:modulusbound}
\limsup_{\delta \to 0} \, \limsup_{n \to \infty} \, \frac{1}{n} \sum_i \E{m^n_i(T, \delta)} = 0.
\end{equation}
Then the sequence of random measures $\{\mu^n\}$ is tight.
If particular, tightness follows if there is an $m : [0,1] \to \R_{+}$
that satisfies $\lim_{\delta \to 0} m(\delta) = m(0) = 0$ and
$$\limsup_{\delta \to 0} \, \limsup_{n \to \infty} \, \max_{i} \, \pr{ m^n_i(T, \delta) > m(\delta)} = 0.$$
\end{lemma}

\begin{proof}
We may replace the limit supremum over $n$ in (\ref{eqn:modulusbound}) by a supremum
over $n$ because for each $n$ and $i$ the expected c\'{a}dl\'{a}g modulus of continuity
$\E{m^n_i(T,\delta)} \searrow 0$ as $\delta \searrow 0$.

The sequence $\{\mu^n\}$ is tight if for every $T < \infty$ the empirical measure of
the restriction of the paths $f^n_1, \ldots, f^n_n$ to $[0,T]$ is tight. Prokhorov's Theorem
\cite[Theorem 14.3]{Kallenberg} provides tightness if for any $\eps > 0$ there exists a compact
$A_{\eps} \subset \M(D([0,T]))$ such that $\sup_n \pr{\mu_n \notin A_{\eps}} \leq \eps$.
A subset $A \subset \M(D([0,T]))$ is compact if for every integer $r \geq 1$ there
is a compact subset $K_r \subset D([0,T])$ such that $\sup_{\mu \in A} \mu(K_r^c) \leq 1/r$.
The compact subsets of $D([0,T])$ in the Skorokhod topology consists of functions with a
common c\'{a}dl\'{a}g modulus of continuity: for any $\eps$ there is a $\delta$ such that
$m_f(\delta) \leq \eps$ for every function $f$ in the subset. We apply these criteria to
construct the required compact subsets $A_{\eps}$.

Due to the assumption in (\ref{eqn:modulusbound}), for fixed $\eps > 0$ there exists $\delta_k \searrow 0$
such that $\sup_n \frac{1}{n} \sum_i \E{m^n_i(T, \delta_k)} \leq \eps/k^4$. Consider the following
compact subset of $D([0,T])$:
$$A_{\eps,r} = \{ f \in D([0,T]) : m_f(\delta_k) \leq r^3/k^2\;\text{for every}\;k\}.$$
Set $$A_{\eps} = \{ \mu \in \M(D([0,T])): \mu(A_{\eps,r}) \geq 1-(1/r)\;\text{for every}\;r\}.$$
Then $A_{\eps}$ has compact closure in $\M(D([0,T],[0,1]))$ and it suffices to show
that for $\pr{\mu^n \notin A_{\eps}} \leq \eps$ for every $n$.
To this end, observe that for any subset $A \subset D([0,T])$ we have $\E{\mu^n(A)} = (1/n) \sum_i \pr{f^n_i \in A}$.
Applying union bounds in $r$ and $k$ respectively, and also Markov's inequality, we get
\begin{align*}
\pr{\mu^n \notin A_{\eps}} & \leq \sum_r \pr{\mu^n(A_{\eps,r}^c) > 1/r} \;\;(\text{union bound over}\;r) \\
& \leq \sum_r r \, \E{\mu^n(A_{\eps,r}^{c})} \;\; (\text{Markov's inequality})\\
& \leq \sum_r \sum_k r\; \frac{1}{n} \sum_i \pr{m^n_i (T, \delta_k) > \frac{r^3}{k^2}} \;\;(\text{union bound over}\;k) \\
& \leq \sum_{r,k} \frac{k^2}{r^2} \; \frac{1}{n} \sum_i \E{m^n_i(T, \delta_k)} \;\;(\text{Markov's ineqaulity})\\
& \leq \eps \,\sum_{r,k} (rk)^{-2}.
\end{align*}

We conclude that $\pr{\mu^n \notin A_{\eps}} \leq B \eps$ for a constant $B$,
which provides tightness of the sequence $\{\mu^n\}$. To deduce the second
statement of the lemma let $p_n(\delta) = \max_i \pr{m^n_i(T, \delta) > m(\delta)}$
for $\delta \leq 1$. Observe that $(1/n) \sum_i \E{m^n_i(T, \delta)} \leq m(\delta) + 2p_n(\delta)$.
Taking limits in $n$ and then $\delta \to 0$ we see that the criterion in (\ref{eqn:modulusbound}) holds.
\end{proof}

\paragraph{Deterministic limits of random measures}

In order to show that the limit of $\nu^n$ is deterministic we will use
the following criterion. Suppose $\nu$ is a random measure in $\M(D(\R_{+}))$.
Then $\nu$ is a $\M(D(\R_{+}))$-valued random variable.
As such, we may first sample an outcome of $\nu$ and then sample two independent
processes, $X$ and $Y$, from said outcome. The joint law of the pair $(X,Y)$ is then
a probability measure on $D(\R_{+},[0,1]^2)$, which we call the law of two samples from $\nu$.
We may also sample an outcome from an independent copy of $\nu$ itself and then
sample a process $Z$ from that outcome. If the joint law of $(X,Y)$ equals the joint law of
$(X,Z)$ then $\nu$ is deterministic. This is the content of the following lemma.

\begin{lemma} \label{lem:concentrationcriterion}
Let $K$ be a compact metric space and $\eta$ be a random measure on $K$.
Let $x,x'$ be two samples from $\eta$ and let $x''$ be a sample from an
independent copy of $\eta$. If the joint law of $(x,x')$ equals the joint law of $(x,x'')$
as $K^2$-valued random variables, then there is a $\mu \in \M(K)$ such
that $\eta \equiv \mu$ almost surely.
\end{lemma}

\begin{proof}
Suppose $\eta$ is defined on a probability space $(\Omega, \Sigma, P)$ and
let $\eta(\omega)$ denote an outcome from $\eta$. As $\eta(\omega)$ is a
Borel probability measure on $K$ we may consider the integral
$I_f(\omega) =  \int f \,d \eta(\omega)$ of a continuous function $f: K \to \R$.
Then $\omega \to I_f(\omega)$ defines a $\R$-valued random variable.
The assumption that the law of two samples from $\eta$ is the same as
the law of samples from two independent copies of $\eta$ implies
that $\E {I_f^2} = \E{I_f}^2$ for every continuous $f$. Therefore, $\E{(I_f-\E{I_f})^2} = 0$.
This implies that there is an $\Omega_f \subset \Omega$ with $P(\Omega_f) =1$
such that $I_f(\omega) = \E{I_f}$ for every $\omega \in \Omega_f$.

Define a deterministic measure $\E{\eta} \in \M(K)$ by the criterion
that for every continuous $f : K \to \R$ we have $\int f \, d \E{\eta} = \E{I_f}$.
This defines the measure by the Reisz representation theorem as $K$ is compact.
Consider a countable dense set of continuous function $\{f_n\}$ from $K \to \R$ with respect
to the topology of uniform convergence. Set $\Omega_{\infty} = \cap_n \Omega_{f_n}$.
Then $P(\Omega_{\infty}) = 1$, and for every $\omega \in \Omega_{\infty}$ we have that
$I_{f_n}(\omega) = \int f_n \,d\E{\eta}$ for every $n$.

We may approximate an arbitrary continuous function by the $f_n$'s.
Thus, we deduce that $I_f(\omega) = \int f\,d\E{\eta}$ for every continuous $f: K \to \R$
so long as $\omega \in \Omega_{\infty}$. From the Reisz representation theorem we conclude
that $\eta(\omega) = \E{\eta}$ for every $\omega \in \Omega_{\infty}$, as required.
\end{proof}

\section{Proof of Theorem \ref{thm:interchangeprocess}} \label{sec:interchange}

The trajectory of each particle in $\Int$ is a continuous time simple random walk on $P_n$.
In order to show convergence of the empirical measure $\nu^n$ of particle trajectories
we will first show that it has subsequential limits by the criterion on Lemma \ref{lem:compactcriterion}.
Any subsequential limit will be a random permuton process with sample paths of class H\"{o}lder$(1,1/8)$
due to the estimate provided in Lemma \ref{lem:compactness} below. Since the set of such functions is compact in
$D(\R_{+})$ we will then verify the criterion of Lemma \ref{lem:concentrationcriterion} to conclude that any
subsequential limit must be a deterministic permuton process. Lemma \ref{lem:concentration} provides
the quantitate estimate used to verify this criterion. Finally, we will identify the limit as stationary Brownian
motion on $[0,1]$ via Donsker's theorem. The most technical part of the argument is the
proof of Lemma \ref{lem:concentration}, which is given in Section \ref{sec:concentrationproof}.

\begin{lemma}[Tightness] \label{lem:compactness}
For $T < \infty$, $\delta > 0$ and any particle trajectory $T^n_i$ of $\Int$,
$$ \pr{\sup_{s, t \in [0,T];\;|t-s| \leq \delta} \, \left |T^n_i(t) - T^n_i(s) \right | > \delta^{1/8}}
\leq 10^3\,T \,\big (\delta^{1/2} + \frac{\delta^{-1/2}}{n^2} \big ).$$

Consequently, the empirical measures $\nu^n$ of the particle trajectories have subsequential
limits and all limit points are random permuton processes with sample paths of class
H\"{o}lder$(1,1/8)$ almost surely.
\end{lemma}

\begin{lemma}[Concentration] \label{lem:concentration}
For any two particles $i \neq j$ there exists coupled random trajectories
$T_1(t)$, $T_2(t)$ and $T_3(t)$ for $t \geq 0$ with the following properties.
The pair $(T_1,T_2)$ has the joint law of $(T^n_i, T^n_j)$,
the trajectories of particles $i$ and $j$ in $\Int$. $T_3$ is independent of $T_1$ and has the
law of $T^n_j$. Finally, for some universal constant $C$,
$$ \pr{\sup_{0 \leq t \leq T} \,|T_2(t) - T_3(t)| > n^{-\frac{1}{4}}} \leq \frac{C\,T^{1/2}}{n^{1/2}}.$$
\end{lemma}

\paragraph{A covering argument}

In order to prove Lemmas \ref{lem:compactness} and \ref{lem:concentration}
we interpret the law of trajectories of particles in $\Int$ as projections of particle
trajectories of the interchange process on $\Z$. Consider the map $\pi_n : \Z \to V(P_n)$
defined as follows. First, map an integer $x$ to $x \;(\mathrm{mod}\;2n)$ with coset
representative chosen from the set $\{1,\ldots, 2n\}$. Second, map the coset representative
$y$ to $y \in V(P_n)$ if $1 \leq y \leq n$, and otherwise, map $y$ to $2n+1 - y \in V(P_n)$.
The map $\pi_n$ is the composition of the first map followed by the second.
Note that $\pi_n$ preserves vertex adjacencies: $\{\pi_n(x), \pi_n(x+1)\}$ is an edge of $P_n$
for every $x \in \Z$. In particular, $\pi_n$ is Lipschitz:
$$|\pi_n(a) - \pi_n(b)| \leq |a-b| \;\;\text{for every}\;\; a, b \in \Z.$$

Consider the continuous time interchange process on $\Z$ where particles move
according to edge firings of independent Poisson processes $\{ \mathrm{Poi}_{\{x,x+1\}}\}$
for $x \in \Z$, each firing at rate $1/2$. Let $S_i(t)$ be the trajectory of particle $i$.
Then $(S_i(t); t \geq 0)$ has the law of continuous time simple random walk
(SRW) on $\Z$ started from $i$. The continuous time SRW may be interpreted at the law
of $\{i + Z_{N(t)}; t \geq 0\}$, where $Z_0 = 0, Z_1, \ldots$ is a discrete time SRW on $\Z$
and $(N(t); t \geq 0)$ is a rate 1 Poisson process on $\R$ that is independent of $Z_0, Z_1, \ldots$.
It is easy to see from this that
$$\E{|S_i(t_2)-S_i(t_1)|^4} = 3(t_2-t_1)^2 + t_2-t_1 \;\;\text{for all times}\;\; 0 \leq t_1 \leq t_2< \infty.$$
Observe that for any finite set of particles
$1 \leq i_1 < \ldots < i_k \leq n$, the joint law of the trajectories $(\Int_t(i_1), \ldots, \Int_t(i_k); t \geq 0)$
is the same as the joint law of $\big(\pi_n(S_{i_1}(t)), \ldots, \pi_n(S_{i_k}(t)); t \geq 0\big)$.

Let $(\mathcal{F}(t); t \geq 0)$ denote the natural filtration for the above interchange process
on $\Z$. Namely, $\mathcal{F}(t)$ is generated by the sigma-algebras
$\mathcal{F}_{\{x,x+1\}}(t) = \sigma( \mathrm{Poi}_{\{x,x+1\}}[0,t])$ over all $x \in \Z$.
Then $(\mathcal{F}(t); t \geq 0)$ is a right continuous filtration and each particle trajectory
$S_i$ is a SRW on $\Z$ adapted to the filtration.


\paragraph{Proof of Lemma \ref{lem:compactness} (Tightness)}
\begin{proof}
The cover map $\pi_n$ allows us to bound the modulus of continuity of $\Int_t(i)$ by way of the
modulus of continuity of $S_i$. For each $i$, the process $S_i(t) - S_i(0)$ is a martingale
w.r.t.~the filtration $(\mathcal{F}(t); t \geq 0)$. By Doob's maximal inequality for matingales
(continuous time with c\`{a}dl\`{a}g paths) we have that that for all particles $i$ and
times $0 \leq t_1 \leq t_2 < \infty$,
\begin{equation} \label{eqn:Doob}
\pr{ \sup_{t_1 \leq s \leq t_2} \, |S_i(s) - S_i(t_1)| > \lambda} \leq \frac{16 \,\E{|S_i(t_2)-S_i(t_1)|^4}}{\lambda^4}
\leq \frac{48\,|t_2-t_1|^2 + 16\,|t_2-t_1|}{\lambda^4}.
\end{equation}

The law of the trajectory $T^n_i$ of $\Int$ is the same at the law of $(\pi_n(S_i(n^2t))/n; t \geq 0)$.
As $\pi_n$ is Lipschitz we have from (\ref{eqn:Doob}) that for all particles $i$, times
$t \geq 0$, and $0 \leq \delta \leq 1$,
\begin{align}
\nonumber \pr{\sup_{t \leq s \leq t + \delta} \, \left |T^n_i(s) - T^n_i(t) \right | > \delta^{1/8}}
&= \pr{ \sup_{t \leq s \leq (t+\delta)} \, \Big | \frac{\pi_n(S_i(n^2s))}{n} - \frac{\pi_n(S_i(n^2t))}{n} \Big | >  \delta^{1/8}} \\
\nonumber &\leq \pr{\sup_{t \leq s \leq (t+\delta)} \, | S_i(n^2s) - S_i(n^2t) | > n \delta^{1/8}} \\
\label{eqn:interchangetightness} &\leq 2^6 \big (\delta^{3/2} + \frac{\delta^{1/2}}{n^2} \big ).
\end{align}

Fix $T < \infty$ and set $m_i(\delta) = \sup_{s,t \in [0,T]: |s-t| \leq \delta} \big |T^n_i(t)-T^n_i(s) \big |$.
Note that $m_i$ dominates the c\'{a}dl\'{a}g modulus of continuity of $T^n_i$ by considering
the partition of $[0,T]$ into equally spaced points of mesh size $\delta$. We have that
$$m_i(\delta) \leq 2 \max_{1 \leq j \leq \lceil T/\delta \rceil} \,\sup_{(j-1)\delta \,\leq \, s,t \, \leq \, j\delta}\,
\Big |T^n_i(t)- T^n_i(s) \Big |.$$
Employing an union bound over $j$ and utilizing (\ref{eqn:interchangetightness}) we deduce that
$$\max_{1 \leq i \leq n} \pr{m_i(\delta) > 2 \delta^{1/8}} \leq
2^6\,\left \lceil \frac{T}{\delta} \right \rceil (\delta^{3/2} + \delta^{1/2}n^{-2}) \leq
2^7 T \,(\delta^{1/2}+ \frac{\delta^{-1/2}}{n^2}).$$

From Lemma \ref{lem:compactcriterion} we conclude that the random measures
$\{\nu^n\}$ with times restricted to $[0,T]$ has subsequential weak limits converging to
random measures on $D([0,T])$. As this holds for every $T$ it follows that $\{\nu^n\}$ has weak
subsequential limits converging to random measure on $D(\R_{+})$.
Any limit point is a random permuton process $(\mathbf{X}(t); 0 \leq t \leq T)$.
The estimate (\ref{eqn:interchangetightness}) readily implies that the outcomes of $\mathbf{X}$
are permuton processes with sample paths of class H\"{o}lder$(1,1/8)$.
This completes the proof of Lemma \ref{lem:compactness}.
\end{proof}

\paragraph{\textbf{Proof of Theorem \ref{thm:interchangeprocess}}}
Lemma \ref{lem:concentration} is proved in Section \ref{sec:concentrationproof} below.
We will finish the proof of Theorem \ref{thm:interchangeprocess} from the two key lemmas.
Lemma \ref{lem:compactness} shows that $\nu^n$ has subsequential limits that are supported
on permuton processes with H\"{o}lder$(1,1/8)$ sample paths. Let $\mathbf{X}$ be such a
limit point, which we will conclude is a deterministic permuton process. By way of Lemma
\ref{lem:concentrationcriterion} it suffices to show that the joint law of two samples from
$\mathbf{X}$ is the same as the law of samples from two independent copies of $\mathbf{X}$.
After reducing to a subsequence we may assume that $\nu^n$ converges to $\mathbf{X}$.

Employing Lemma \ref{lem:concentration}, for every $1 \leq i \leq n$ and $j \neq i$ we find coupled trajectories
$(T^n_i,T^n_j,T^n_{i,j})$ such that $(T^n_i,T^n_j)$ has the joint law of the trajectories of
particles $i$ and $j$ in $\Int$ while $T^n_{i,j}$ has the law of $T^n_j$ and is independent
of $T^n_i$. We also have that for some universal constant $C$,
\begin{equation} \label{eqn:supestimate}
\pr{\Big | \Big | T^n_j(t) - T^n_{i,j}(t) \Big| \Big |_{\mathrm{L}^{\infty}[0,T]} > \frac{2}{n^{1/4}}}
\leq \frac{C \sqrt{T}}{\sqrt{n}}.\end{equation}

For every $1 \leq i , j \leq n$ set $X_i = T^n_i$, $Y_j = T^n_j$ if $j \neq i$
and $Y_i = X_i$, and $Z_{i,j} = T^n_{i,j}$ if $i \neq j$ and $Z_{i,i}$ an independent
copy of $X_i$. Let $I, J \sim \uni{\{1,\ldots, n\}}$ be independent random indices
that are also independent of $\Int$. We omit writing the dependence of $I$ and
$J$ on $n$. The pair $(X_I, Y_J)$ converges to the law of two samples from $\mathbf{X}$ and
the pair $(X_I, Z_{I,J})$ converges to the law of samples from two independent
copies of $\mathbf{X}$. It suffices to show that $|| (X_I,Y_J) - (X_I, Z_{I,J}) ||_{\mathrm{L}^{\infty}[0,T]}$
converges to zero in probability to conclude that $\mathbf{X}$ is deterministic.
To that end, note that

\begin{align*}
\pr{||Y_J - Z_{I,J}||_{\mathrm{L}^{\infty}[0,T]} > 2n^{-1/4}} &\leq \pr{||Y_J - Z_{I,J}||_{\mathrm{L}^{\infty}[0,T]} > 2n^{-1/4}; I \neq J} + \frac{1}{n}\\
&= \frac{1}{n^2-n} \sum_{i \neq j} \pr{||Y_j - Z_{i,j}||_{\mathrm{L}^{\infty}[0,T]} > 2n^{-1/4}} \;+ \frac{1}{n}\\
&\leq \max_{i \neq j} \; \pr{||Y_j - Z_{i,j}||_{\mathrm{L}^{\infty}[0,T]} > 2n^{-1/4}} \;+ \frac{1}{n}\\
&\leq \frac{C \sqrt{T}}{\sqrt{n}} + \frac{1}{n}.
\end{align*}

This concludes the proof that all limit points of $\nu^n$ are deterministic.
Also, the limit points are the limit points of $X_I$, that is, the weak limit points
of the trajectory of a random particle considered as a random function in
$D(\R_{+})$. The law of $X_I$ is the same as that of the process
$$\frac{2}{n}\, \pi_n(I + S(n^2t)) - 1\;\; \text{for}\;\; t \geq 0,$$
where $t \to S(t)$ is SRW on $\Z$ started from the origin and independent of $I$.

Consider the map $\pi : \R \to [0,1]$ defined as follows. First, map any
real number $x$ to its unique coset representative $y$ in $\R / 2\Z$ with $y \in (0, 2]$.
Then map $y$ to itself if $0 < y \leq 1$, and otherwise, map $y$ to $2-y \in [0,1)$.
The map $\pi$ is Lipschitz and $\big| \frac{\pi_n(k)}{n} - \pi(\frac{k}{n}) \big | \leq \frac{1}{n}$ for
every $k \in \Z$. This estimate between $\pi_n$ and $\pi$ implies that
$$ \Bigg | \Bigg | \Big (\frac{\pi_n(I+S(n^2t))}{n}\Big) - \Big(\pi \big( \frac{I+S(n^2t)}{n}\big)\Big)
\Bigg | \Bigg |_{\mathrm{L}^{\infty}[0,T]} \leq \frac{1}{n}.$$

The estimate above implies that any limit point of $X_I$ must be stationary Brownian motion
on $[0,1]$ so long as the process $\big (\pi(\frac{I+S(n^2t)}{n}); t \geq 0 \big)$ converges to
stationary Brownian motion on $[0,1]$. Donsker's Theorem implies that the process
$\big ( \frac{S(n^2t)}{n}; t \geq 0 \big)$ converges weakly on $D(\R_{+},\R)$ to standard
Brownian motion $(B(t); t \geq 0)$. As $I$ is independent of $S$ we deduce that the process
$$\Big (\, \frac{I+S(n^2t)}{n}; \, t \geq 0 \, \Big) \to (U + B(t); t \geq 0)\;\;\text{weakly},$$
where $U \sim \Uni$ is independent of $B$. Since $\pi$ is Lipschitz and thus continuous,
we have that $\big (\pi \left (\frac{I+S(n^2t)}{n}\right); t \geq 0 \big)$ converges to $\big (\pi(U + B(t)); t \geq 0 \big )$.
The latter process is stationary Brownian motion on $[0,1]$, which is Brownian motion stated
from an uniform random point in $[0,1]$ and reflected off of the lines $y=0$ and $y=1$.

We thus conclude that $\{\nu^n\}$ is tight and any limit point $\mathbf{X}$ must be stationary
Brownian motion on $[0,1]$. This implies that $\{\nu^n\}$ converges to stationary Brownian
motion on $[0,1]$.

\subsection{Concentration of the interchange process: Proof of Lemma \ref{lem:concentration}} \label{sec:concentrationproof}

We will prove Lemma \ref{lem:concentration} by utilizing the covering of $\Int$ by the interchange
process on $\Z$ introduced earlier. We will first construction appropriate coupled trajectories
for the interchange process on $\Z$ and then project them to $\Int$ using the covering map $\pi_n$.

Fix particles $1 \leq i < j \leq n$ and let $S_1$ and $S_2$ denote the trajectories $S_i$ and $S_j$
of particles $i$ and $j$ for the interchange process on $\Z$, respectively.
We define a new trajectory $S_3$ of particle $j$ that will be independent of $S_1$
but coupled with $S_2$ so that $||S_2-S_3||_{L^{\infty}[0,T]}$ is of order $\sqrt{T}$.
Roughly speaking, $S_3$ will take the same steps as $S_2$ except between times
when $|S_1 - S_2| = 1$, where $S_3$ will move independently as a SRW on $\Z$.
More precisely, consider the stopping times $0 = \tau_0^{+} < \tau_1 < \tau_1^{+} < \cdots$, where
$$\tau_j = \inf \{ t \geq \tau_{j-1}^{+} : |S_1(t)-S_2(t)| = 1\} \quad \text{and}
\quad \tau_j^{+} = \inf \{ t \geq \tau_j : |S_1(t)-S_2(t)| > 1\}.$$
As SRW on $\Z$ is recurrent the $\tau_j$s and $\tau_j^{+}$s are all finite and satisfy
$\tau_1 < \tau_{1}^{+} < \tau_2 < \cdots$ with probability 1. They are also strict stopping times
in the sense that the events $\{\tau_j \leq t\}$ and $\{\tau_j^{+} \leq t\}$ are $\mathcal{F}(t)$-measurable.
Set $\mathcal{C} = \cup_j [\tau_j, \tau_j^{+})$. Then $|S_1(t) - S_2(t)| = 1$ precisely for
times $t \in \mathcal{C}$.

In order to define the process $S_3$ we introduce a new collection of independent Poisson processes
$\{\mathrm{Poi}'_{\{x,x+1\}}\}$ for $x \in \Z$ that are also independent of the processes
$\{\mathrm{Poi}_{\{x,x+1\}}\}$ used to define the interchange process on $\Z$.
We may assume that the edge firings of the new Poisson processes
and the old one are all distinct. Let $S_3(0) = j$, and for times $t \in \mathcal{C}$ let
the trajectory $S_3$ move according to the edge firings of $\{\mathrm{Poi}'_{\{x,x+1\}}\}$.
More precisely, if $t \in \mathcal{C}$ and the edge $\{x,x+1\}$ fires at time $t$ in $\mathrm{Poi}'_{\{x,x+1\}}$
with $S_3(t-) = x$ then let $S_3(t) = x+1$. For other times $t$ the particle at $S_3(t)$ moves according
to the direction of the particles at $S_2(t)$. Namely, if $S_2(t)$ changes to $S_2(t) +1$ or $S_2(t)-1$
after encountering an edge firing at time $t$ then $S_3(t)$ also moves to $S_3(t)+1$ or $S_3(t)-1$, respectively,
in accordance to $S_2(t)$'s direction. In particular, $|S_3(t)-S_2(t)|$ changes only for times $t \in \mathcal{C}$.
The process $S_3$ is independent of $S_1$ and the law of $S_3$ is that of a SRW on $\Z$ stared from $j$.

\paragraph{\textbf{Deviation between $S_2$ and $S_3$}}
The processes $S_2(t)$ and $S_3(t)$ are both martingales with respect to the natural filtration
induced by $\{\mathrm{Poi}_{\{x,x+1\}}, \mathrm{Poi}'_{\{x,x+1\}}\}_{x \in \Z}$.
Therefore, $S_3(t) - S_2(t)$ is also a martingale with $S_3(0)-S_2(0)=0$. Doob's maximal inequality gives
\begin{equation} \label{eqn:longtimebound}
\pr{\sup_{0 \leq t \leq T} |S_3(t)-S_2(t)| > \lambda} \leq \frac{4}{\lambda^2}\, \E{|S_3(T)-S_2(T)|^2}.
\end{equation}
The rest of the argument shows that $\E{|S_3(T)-S_2(T)|^2}$ is of order $O(\sqrt{T})$.

Consider the random variable
\begin{equation} \label{eqn:J} J = \inf \{ j \geq 1 : \tau_j^{+} > T\}.\end{equation}
The number of times particles $i$ and $j$ are distance 1 apart within the time interval $[0,T]$ is
either $J$ or $J-1$. The random variable $J$ is $\mathcal{F}(T)$-measurable
because the event $\{ J \geq j\}$ equals $\mathcal{F}(T)$-measurable event
that particles $i$ and $j$ have been at distance 1 from each other at least $j$ times
within the time period $[0,T]$. Moreover, the event $\{J \geq j\}$ is $\mathcal{F}(\tau_j)$-measurable
because $\{J \geq j\} \cap \{\tau_j \leq t\}$ is the event that particles $i$ and $j$ have been
distance 1 apart at least $j$ times in the time interval $[0, \min \{t,T\}]$.

\begin{lemma} \label{lem:returns}
For $T \geq 1$, the expectation $\E{J} \leq 10 \sqrt{T}$.
\end{lemma}

\begin{proof}
Consider the process $Z$ defined by $Z(t) = |S_2(t)-S_1(t)| - 1$, which takes values in
$\Z_{\geq 0} = \{0,1,2,\ldots\}$. The law of $Z$ is of a random walk on $\Z_{\geq 0}$
with edges firing as follows. The edges $\{x,x+1\}$ for $x \geq 0$ fire at rate 1 and
there is a self loop at $\{0\}$ that fires at rate $1/2$. Let $V$ be the number of times $Z$
visits the origin $\{0\}$ in the time interval $[0,T]$. Observe that $J \leq V + 1$ because the times
when $Z(t) = 0$ are the times when particles $i$ and $j$ are distance 1 apart in the interchange
process on $\Z$.

Consider the SRW on $\Z_{\geq 0}$ with all edges edges $\{x,x+1\}$ firing at rate 1 and also
a self loop at $\{0\}$ firing at rate 1. If we start the walk at $j-i-1$ then $V$ is stochastically
dominated by the number of visits to the origin of this SRW in the time period $[0,T]$. Indeed,
the edge firings of the self loop for the process $Z$ is stochastically dominated by the edge
firings of this SRW in the sense that the two processes can be coupled such that the edge firings
of the former can be made to be contained in the edge firings of the latter. This follows because
if every point of a Poisson process with rate 1 is deleted independently with probability $1/2$
then it results in a Poisson process of rate $1/2$ (thinning property of Poisson processes).
Hence, it suffices to bound the expected number of visits to the origin in time period $[0,T]$ of a
SRW on $\Z_{\geq 0}$ with all edges firing as rate 1.

We may cover the above graph on $\Z_{\geq 0}$, with the self loop, via a map $\hat{\pi} : \Z \to \Z_{\geq 0}$
that sends the law of SRW on $\Z$ to the law of SRW on $\Z_{\geq 0}$. This is entirely analogous
to the covering maps $\pi_n: \Z \to P_n$ introduced earlier. Visits to $\{0\}$ of the SRW on
$\Z_{\geq 0}$ then corresponds to visits to the set $\{0,1\}$ for the SRW on $\Z$.
The expected number of visits of SRW on $\Z$, started from any vertex, to the set $\{0,1\}$
in time period $[0,T]$ is at most $6 \sqrt{T}$ by Lemma \ref{lem:SRWvisits}. This implies that $\E{J} \leq 10 \sqrt{T}$.
\end{proof}

We now proceed to bound $\E{|S_3(T)-S_2(T)|^2}$.
As $S_3(t)-S_2(t)$ changes only for times $t \in \mathcal{C}$, we have
\begin{align} \label{eqn:displacement}
S_3(T) - S_2(T) &= \sum_{j} \,\left [(S_3(\tau_j^{+}) - S_3(\tau_j)) - (S_2(\tau_j^{+}) - S_2(\tau_j))\right] \cdot \ind{j < J} + \\
\nonumber &+ \left [(S_3(T) - S_3(\tau_J)) - (S_2(T) - S_2(\tau_J)) \right] \cdot \ind{T > \tau_J}.
\end{align}
We consider the distributions of the random variables in the summands above
in order to calculate their mean and variance.

We claim that $[S_2(\tau_j^{+}) - S_2(\tau_j)] \cdot \ind{j < J}$  is $\mathcal{F}(T)$-measurable
for every $j$. Indeed, the event $\{j < J\}$ implies $\{\tau_j^{+} \leq T\}$, which then implies that
$$[S_2(\tau_j^{+}) - S_2(\tau_j)]\ind{j < J} =  [S_2(\tau_j^{+}) - S_2(\tau_j)] \ind{\tau_j^{+} \leq T} \ind{j < J}.$$
The r.h.s.~above is $\mathcal{F}(T)$-measurable as both
$[S_2(\tau_j^{+}) - S_2(\tau_j)] \ind{\tau_j^{+} \leq T}$ and $\ind{j < J}$ are $\mathcal{F}(T)$-measurable.
Similarly, $(S_2(T) - S_2(\tau_J))\ind{T > \tau_J}$ is $\mathcal{F}(T)$-measurable. Indeed, observe
that $\{T > \tau_J\}$ is an $\mathcal{F}(T)$-measurable event since it means that particles $i$ and $j$
are distance 1 apart at time $T$. 

Now consider the distribution of $[S_3(\tau_j^{+}) - S_3(\tau_j)] \ind{j < J}$ conditional on $\mathcal{F}(T)$.
By design, this conditional distribution is the displacement of a SRW on $\Z$ from time 0 to time $\tau_j^{+}-\tau_j$.
Hence,
\begin{align*}
\E{[S_3(\tau_j^{+}) - S_3(\tau_j)] \ind{j < J}\mid \mathcal{F}(T)} &= 0 \quad \text{and} \\
\E{|S_3(\tau_j^{+}) - S_3(\tau_j)|^2 \ind{j < J} \mid \mathcal{F}(T)} &= (\tau_j^{+} - \tau_j)\ind{j < J}.
\end{align*}
Similarly, $[S_3(T)-S_3(\tau_J)]\ind{T > \tau_J}$ has mean 0 and variance $(T-\tau_J)\ind{T > \tau_J}$
conditional on $\mathcal{F}(T)$.

In order to calculate $\E{|S_3(T)-S_2(T)|^2}$ we first condition over
$\mathcal{F}(T)$, use the representation (\ref{eqn:displacement}), and use the observations above
about the conditional expectations given $\mathcal{F}(T)$ to conclude that
\begin{align} \label{eqn:squaredbound}
\E{|S_3(T)-S_2(T)|^2} &= \E{\sum_{j< J} (\tau_j^{+}-\tau_j) + (T-\tau_J)\ind{T > \tau_J}}\\
\nonumber &+ \E{\left (\sum_{j < J} S_2(\tau_j^{+})-S_2(\tau_j) + (S_2(T)-S_2(\tau_J))\ind{T > \tau_J} \right)^2}.
\end{align}
In the following two lemmas we bound the first and second expectations on the r.h.s.~of (\ref{eqn:squaredbound}).

\begin{lemma} \label{lem:firstexpectation}
$\E{\sum_{j< J} (\tau_j^{+}-\tau_j) + (T-\tau_J)\ind{T > \tau_J}} \leq 10 \sqrt{T}$.
\end{lemma}

\begin{proof}
Define
$$M_n = \sum_{j=1}^n (\tau_j^{+}-\tau_j - 1).$$
We make use of the strong Markov property of the process
$\vec{S} = \big ((S_1(t),S_2(t)); t \geq 0\big)$.
By the strong Markov property the $\tau_j^{+}-\tau_j$ are independent and
identically distributed because $\tau_j^{+}-\tau_j$ is a stopping time for
$(\vec{S}(t+ \tau_j)-\vec{S}(\tau_j))$. Moreover, $\tau^{+}_1-\tau_1$ is distributed as an
exponential random variable of rate 1. Indeed, its distribution is the first time particles
$0$ and $1$ from the interchange process on $\Z$ move farther than distance 1 apart,
which is the minimum of two independent exponential random variables of rate $1/2$.
These observations imply that $\{M_n\}$ is a mean zero martingale with respect to the
filtration $(\mathcal{F}(\tau^{+}_n); n \geq 0)$. The random variable $J$ is a stopping time
for this filtration since the event
$\{J \geq j\} \in \mathcal{F}(\tau_j) \subset \mathcal{F}(\tau^{+}_j)$. Assume that $T \geq 1$.
As $\E{J} \leq 10 \sqrt{T}$ by Lemma \ref{lem:returns}, the Optional Stopping Theorem implies $\E{M_J} = 0$.
As $(T-\tau_J)\ind{T > \tau_J} \leq \tau_J^{+} - \tau_J$, we have
$$\sum_{j < J} (\tau_j^{+}-\tau_j) + (T-\tau_J)\ind{T > \tau_J} \leq M_J + J.$$
Therefore,
$$\E{\sum_{j < J} (\tau_j^{+}-\tau_j) + (T-\tau_J)\ind{T > \tau_J}} \leq \E{J} \leq 10 \sqrt{T}.$$
\end{proof}

\begin{lemma} \label{lem:secondexpectation}
$$ \E{\left (\sum_{j < J} S_2(\tau_j^{+})-S_2(\tau_j) + (S_2(T)-S_2(\tau_J))\ind{T > \tau_J} \right)^2} \leq 100 \sqrt{T}.$$
\end{lemma}

\begin{proof}
Note that $|S_2(T)-S_2(\tau_J)|\ind{T > \tau_J} \leq 1$. From the inequality $(a+b)^2 \leq 2(a^2 + b^2)$
applied with $b = S_2(T)-S_2(\tau_J)\,\ind{T > \tau_J}$ we deduce that
\begin{equation} \label{eqn:secondterm}
\left (\sum_{j < J} S_2(\tau_j^{+})-S_2(\tau_j) + (S_2(T)-S_2(\tau_J))\ind{T > \tau_J} \right)^2
\leq 2 \left (\sum_{j < J} S_2(\tau_j^{+})-S_2(\tau_j) \right )^2 + 2.\end{equation}

We claim that $\left \{\sum_{j \leq n} (S_2(\tau_j^{+})-S_2(\tau_j)) \right \}_{n \geq 1}$
is a mean zero Martingale adapted to the filtration $\{\mathcal{F}(\tau_n^{+})\}$ with the
increment $S_2(\tau_j^{+})-S_2(\tau_j)$ being independent of $\mathcal{F}(\tau_{j-1}^{+})$.
To see this note that the strong Markov property of the process $\vec{S} = (S_1,S_2)$ implies that
$S_2(\tau_j^{+})-S_2(\tau_j)$ is independent of $\mathcal{F}(\tau_{j-1}^{+})$.
The Optional Stopping Theorem gives $\E{S_2(\tau_j^{+})-S_2(\tau_j)} = 0$ because
$\tau_j^{+}-\tau_j$ is a stopping time for the process $(\vec{S}(t+ \tau_j)-\vec{S}(\tau_j); t \geq 0)$
and it has finite expectation. Finally, observe that $|S_2(\tau_j^{+})-S_2(\tau_j)| \leq 2$.
Indeed, suppose two particles $a$ and $b$ are distance 1 apart on the interchange process
during the time interval $[t_1,t_2)$. If they drift farther apart at time $t_2$ then the absolute
displacement of either $a$ or $b$ from its respective position at time $t_1$ to its position at time
$t_2$ is at most 2. These three observations imply the claim. In particular,
$$M'_n = \Big (\sum_{j \leq n} S_2(\tau_j^{+})-S_2(\tau_j) \Big)^2 -\sum_{j \leq n} \E{(S_2(\tau_j^{+})-S_2(\tau_j))^2}$$
is also a Martingale.  The Optional Stopping Theorem now implies $\E{M'_{J-1}} = 0$.
Combining this with the estimate (\ref{eqn:secondterm}) we get that
\begin{align*}
\E{\left (\sum_{j < J} S_2(\tau_j^{+})-S_2(\tau_j) + (S_2(T)-S_2(\tau_J))\ind{T > \tau_J} \right)^2} &\leq \\
2 \E{ \Big (\sum_{j < J} S_2(\tau_j^{+})-S_2(\tau_j) \Big )^2} + 2 &= \\
2\,\E{\sum_{j < J} \big (S_2(\tau_j^{+})-S_2(\tau_j) \big )^2} + 2 &\leq \\
8\, \E{J-1} +2 & \leq 100 \sqrt{T}.
\end{align*}
\end{proof}

From (\ref{eqn:squaredbound}), Lemma \ref{lem:firstexpectation} and Lemma \ref{lem:secondexpectation},
we conclude that $\E{|S_3(T)-S_2(T)|^2} \leq C \sqrt{T}$ for some universal constant $C$ so long as $T \geq 1$.
Then (\ref{eqn:longtimebound}) implies that
$$\pr{\sup_{0 \leq t \leq T} |S_3(t)-S_2(t)| > \lambda} \leq \frac{C\, \max\{T^{1/2},1\}}{\lambda^2}.$$
After applying the cover map $\pi_n$, we get that the processes $\pi_n(S_1), \pi_n(S_2)$ and $\pi_n(S_3)$
have the property that the first two have the law of the non-rescaled trajectories of particles $i$ and $j$
in $\Int$ while the third has the law of the non-rescaled trajectory of particle $j$ and is independent of the first.
As $\pi_n$ is Lipschitz we get that
\begin{equation} \label{eqn:finalestimate}
\pr{ \sup_{0 \leq t \leq n^2T} |\pi_n(S_2(t))-\pi_n(S_3(t))| > \lambda} \leq
\frac{C\, \max\{T^{1/2},1\}\,n}{\lambda^2}.\end{equation}
Setting $T_k(t) = \pi_n(S_k(t))/n$ for $k = 1,2,3$ and $\lambda = n^{3/4}$
in (\ref{eqn:finalestimate}) we get the conclusion of Lemma \ref{lem:concentration}.

\begin{lemma}[SRW visits to origin] \label{lem:SRWvisits}
Let $S = (S(t); t \geq 0)$ be continuous time SRW on $\Z$ where all edges fire at rate 1
and $S(0)= x$ for some fixed $x \in \Z$. Let $V(T)$ be the number of visits of $S$ to the origin
in the time period $[0,T]$. Then $\E{V(T)} \leq 3 \sqrt{T}$. In particular, the
expected number of visits to a set of vertices $\{y_1, \ldots, y_j\}$ within the time interval $[0,T]$
is at most $3j \sqrt{T}$.
\end{lemma}

\begin{proof}
Let $S_0 = x, S_1, S_2, \ldots$ be a discrete time SRW on $\Z$. Let $P$ be a Poisson process
of rate 2 that is independent of $S_0,S_1, \ldots$. Let $N(t)$ be the number of points of $P$
in $[0,t]$. The law of the continuous time SRW on $\Z$ with edges firing at rate $1$ is $(S_{N(t)}; t \geq 0)$.
Let $V_n$ be the number of visits of $S_0, S_1, \ldots$ to the origin from time 0 to time $n$.
Then $V(T)$ equals $V_{N(T)}$ in distribution. If we show that $\E{V_n} \leq 2 \sqrt{n}$, then $N(T)$
being independent of $V_0, V_1, \ldots$ implies that
$$\E{V(T)} = \E {\E{ V_{N(T)} \mid N(T)}} \leq 2 \E{\sqrt{N(T)}} \leq 2 \E{N(T)}^{1/2}
= 2\sqrt{2} \, \sqrt{T}, \;\text{as required}.$$

Note that $\E{V_n} = \sum_{k=0}^n \pr{S_k = 0}$. We may assume that $x \geq 0$ due to symmetry.
$\pr{S_k=0}$ is non-zero only if $k$ and $x$ have the same parity and $k \geq x$. Then,
$$\pr{S_k = 0} = \binom{k}{(k+x)/2} 2^{-k} \leq \binom{k}{\lfloor \frac{k}{2} \rfloor}2^{-k} \leq \frac{1}{\sqrt{k}}.$$
The last estimate follows from Stirling's approximation. Therefore,
$$\E{V_n} \leq 1 + \sum_{k=1}^n \frac{1}{\sqrt{k}} \leq 2\sqrt{n}\,.$$
The last statement of the lemma follows from observing that the number of visits to a vertex $y$ of a
SRW started from $x$ has the same law as the number of visits to the origin of a SRW started from $x-y$.
\end{proof}

\section{The hydrodynamic limit} \label{sec:hydrodynamic}

The symmetric simple exclusion process (SSEP) is a particle system on $P_n$ evolving as follows \cite[chapter 2.2]{KL}.
There are two types of particles - black particles, or occupied sites, and white particles, also called free sites.
The initial configuration of black particles is some subset $\eta_0 \in \{0,1\}^{V(P_n)}$, where $\eta_0(i) =1$
indicates that there is a black particle at vertex $i$. The particles then swap positions according to the dynamics of the
interchange process. If $\eta_t$ is the configuration of black particles at time $t$ then SSEP is the
$\{0,1\}^{V(P_n)}$-valued process $(\eta_t; t \geq 0)$.
Let $|\eta_0| = \sum_i \eta_0(i)$ be the total number of black particles in the SSEP.
The empirical measure of SSEP at time $t$ is the following random probability measure on $[0,1]$:
\begin{equation} \label{eqn:SSEPdensity}
\rho_n(t) = \frac{1}{|\eta_0|} \,\sum_{i=1}^n \delta_{\{\frac{i}{n}\}} \cdot \eta_{n^2t}(i).\end{equation}

The measure $\rho_n(t)$ is related to the particle trajectories of the interchange process.
Let $T^n_j$ be the rescaled trajectory of particle $j$ in $\Int$. The vertices occupied by the black
particles in the SSEP at time $n^2t$ are $\{ n \cdot T^n_j(t)\}_j$ for only those values
of $j$ such that $\eta_0(j) = 1$. Therefore,
$$\rho_n(t) = \frac{1}{|\eta_0|} \,\sum_{i=1}^n \delta_{T^n_j(t)} \cdot \eta_0(j).$$

Let $J_n \sim \uni{\{j: \eta_0(j) =1\}}$ denote a random choice of black particle.
Conditional on $\Int$ the measure $\rho_n(t)$ is the law of $T^n_{J_n}(t)$ over the random $J_n$.
The measure valued process $(\rho_n(t); t \geq 0)$ is the law of the trajectory of a random
black particle in $\Int$. The argument of the previous sections implies that $(\rho_n(t); t \geq 0)$
has the same weak limit as the limiting law of the process $\Big (\pi( \frac{J_n + S(n^2t)}{n}); t \geq 0 \Big)$
considered as a random function in $D(\R_{+})$. (Recall that $(S(t); t \geq 0)$ is SRW on $\Z$
that is independent of all the $J_n$s, and $\pi$ is the covering map from $\R$ onto $[0,1]$.)

Suppose that the law of $J_n/n$, which is $\rho(0)$ rescaled onto $[0,1]$, converges
weakly to the law of a random variable $X_0 \in [0,1]$. As each $J_n$ is independent
of the SRW $S$, the law of the process
$$\left (\pi \left ( \frac{J_n + S(n^2t)}{n} \right); t \geq 0 \right) \to
\Big (\pi \big (X_0 + B(t) \big); t \geq 0 \Big) \;\;\text{weakly in}\;\;D(\R_{+}),$$
where $X_0$ is independent of standard Brownian motion $B$. The latter process
has the law of standard Brownian motion started from $X_0$ and then reflected off
of the lines $y = 0$ and $y=1$.

Consequently, for any $t$, the random measure $\rho_n(t)$ converges weakly to the deterministic
measure induced by the law of $\pi \big (X_0 + B(t) \big)$. This is the hydrodynamic
limit of SSEP, describing the limiting empirical density of the black particles \cite[chapter 4]{KL}.

\section*{Acknowledgements}
M.~Rahman's research was partially supported by an NSERC PDF grant.
B.~Vir\'{a}g's research was partially supported by the NSERC Discovery Accelerator Supplements Program.

\bibliographystyle{habbrv}

\bibliography{refs.bib}

\bigskip

\noindent \small{Mustazee Rahman, Department of Mathematics, Massachusetts Institute of Technology, Cambridge, MA 02139}\\
\noindent E-mail: \texttt{mustazee@mit.edu} \quad Web: \url{www.math.mit.edu/~mustazee}

\vskip 0.05in

\noindent \small{B\'{a}lint Vir\'{a}g, Department of Mathematics, University of Toronto, Toronto, ON Canada M5S 2E4.}\\
\noindent E-mail: \texttt{balint@math.toronto.edu} \quad Web: \url{www.math.toronto.edu/balint}

\end{document}